\documentclass{amsart}
\usepackage{amsmath}
\usepackage{amssymb}
\usepackage{amsthm}
\usepackage{epsfig}
\usepackage{amsfonts}
\usepackage{amscd}
\usepackage{mathrsfs}
\usepackage{enumerate}
\usepackage{latexsym}
\usepackage{url}

\newtheorem{theorem}{Theorem}[section]
\newtheorem{lemma}[theorem]{Lemma}

\newtheorem{question}[theorem]{Question}

\theoremstyle{definition}
\newtheorem{definition}[theorem]{Definition}
\newtheorem{example}[theorem]{Example}

\theoremstyle{remark}
\newtheorem{remark}[theorem]{Remark}
\numberwithin{equation}{section}

\begin{document}


\title[Connectedness modulo a topological property]{Connectedness modulo a topological property}

\author{M.R. Koushesh}
\address{Department of Mathematical Sciences, Isfahan University of Technology, Isfahan 84156--83111, Iran}
\address{School of Mathematics, Institute for Research in Fundamental Sciences (IPM), P.O. Box: 19395--5746, Tehran, Iran}
\email{koushesh@cc.iut.ac.ir}
\thanks{This research was in part supported by a grant from IPM (No. 90030052).}




\begin{abstract}
Let ${\mathscr P}$ be a topological property. We say that a space $X$ is {\em ${\mathscr P}$--connected} if there exists no pair $C$ and $D$ of disjoint cozero--sets of $X$ with non--${\mathscr P}$ closure such that the remainder $X\backslash(C\cup D)$ is contained in a cozero--set of $X$ with ${\mathscr P}$ closure. If ${\mathscr P}$ is taken to be ``being empty" then ${\mathscr P}$--connectedness coincides with connectedness in its usual sense. We characterize completely regular ${\mathscr P}$--connected spaces, with ${\mathscr P}$ subject to some mild requirements. Then, we study conditions under which unions of ${\mathscr P}$--connected subspaces of a space are ${\mathscr P}$--connected. Also, we study classes of mappings which preserve ${\mathscr P}$--connectedness. We conclude with a detailed study of the special case in which ${\mathscr P}$ is pseudocompactness. In particular, when ${\mathscr P}$ is pseudocompactness, we prove that a completely regular space $X$ is ${\mathscr P}$--connected if and only if $\mbox{cl}_{\beta X}(\beta X\backslash\upsilon X)$ is connected, and that ${\mathscr P}$--connectedness is preserved under perfect open continuous surjections. We leave some problems open.
\end{abstract}

\maketitle


\section{Introduction}

Let ${\mathscr P}$ be a topological property. We say that a space $X$ is {\em ${\mathscr P}$--connected} if there exists no pair $C$ and $D$ of disjoint cozero--sets of $X$ with non--${\mathscr P}$ closure such that the remainder $X\backslash(C\cup D)$ is contained in a cozero--set of $X$ with ${\mathscr P}$ closure. If ${\mathscr P}$ is taken to be ``being empty" (i.e., a space has ${\mathscr P}$ if and only if it is empty) then ${\mathscr P}$--connectedness coincides with connectedness in its usual sense. Thus, ${\mathscr P}$--connectedness may be considered as a generalization of connectedness. Our purpose in this article is to study ${\mathscr P}$--connectedness and to see to what extent the standard theorems about connected spaces remain valid in this generalized context. We begin by characterizing completely regular ${\mathscr P}$--connected spaces. Here ${\mathscr P}$ is subject to some mild requirements and ranges over a wide class of topological properties, including almost all covering properties (i.e., topological properties described in terms of the existence of certain kinds of open subcovers or refinements of a given open cover of a certain type). In the special case in which ${\mathscr P}$ is compactness, it follows that a locally compact Hausdorff space $X$ is ${\mathscr P}$--connected if and only if $\beta X\backslash X$ is connected. (Thus, in particular, if ${\mathscr P}$ is compactness, the Euclidean space $\mathbb{R}^n$ is ${\mathscr P}$--connected if and only if $n\geq 2$.) Then, we study conditions under which unions of ${\mathscr P}$--connected subspaces of a given space are ${\mathscr P}$--connected. Also, we study classes of mappings which preserve ${\mathscr P}$--connectedness. We conclude our work with a detailed study of the special case in which ${\mathscr P}$ is pseudocompactness. In particular, in the case when ${\mathscr P}$ is pseudocompactness, we prove that a completely regular space $X$ is ${\mathscr P}$--connected if and only if $\mbox{cl}_{\beta X}(\beta X\backslash\upsilon X)$ is connected, and that ${\mathscr P}$--connectedness is preserved under perfect open continuous surjections. We do not know whether products of ${\mathscr P}$--connected spaces is ${\mathscr P}$--connected; we leave this as an open problem.

This work has been motivated by our previous work \cite{Ko3}. (See also \cite{Ko1}, \cite{Ko2} and \cite{Ko8}.)

We now review briefly some known facts and terminologies. Additional information may be found in \cite{E}, \cite{GJ} and \cite{PW}.

Let ${\mathscr P}$ be a topological property. Then
\begin{itemize}
  \item ${\mathscr P}$ is {\em closed hereditary}, if any closed subspace of a space with ${\mathscr P}$, also has ${\mathscr P}$.
  \item ${\mathscr P}$ is {\em preserved under finite} ({\em countable}, respectively) {\em sums of closed subspaces}, if any space which is expressible as a finite (countable, respectively) union of its closed subspaces each having ${\mathscr P}$, also has ${\mathscr P}$.
  \item ${\mathscr P}$ is {\em invariant under a class of mappings ${\mathscr M}$} ({\em  inverse invariant under a class of mappings ${\mathscr M}$}, respectively) if for any $f\in{\mathscr M}$, where $f:X\rightarrow Y$ is surjective, the space $Y$ ($X$, respectively) has ${\mathscr P}$ provided that $X$ ($Y$, respectively) has ${\mathscr P}$.
\end{itemize}

Let $X$ be a space and let ${\mathscr P}$ be a topological property. The space $X$ is called a {\em ${\mathscr P}$--space} if it has ${\mathscr P}$. A {\em ${\mathscr P}$--subspace} of $X$ is a subspace of $X$ which has ${\mathscr P}$. By a {\em ${\mathscr P}$--neighborhood} of a point in $X$  we mean a neighborhood of the point in $X$ having ${\mathscr P}$. The space $X$ is called {\em locally--${\mathscr P}$} if each of its points has a ${\mathscr P}$--neighborhood in $X$. Note that if $X$ is regular and ${\mathscr P}$ is closed hereditary, then $X$ is locally--${\mathscr P}$ if and only if each $x\in X$ has an open neighborhood $U$ in $X$ such that $\mbox{cl}_XU$ has ${\mathscr P}$.

Let $X$ and $Y$ be spaces. A mapping $f:X\rightarrow Y$ is called {\em perfect}, if $f$ is closed (not necessarily surjective) and continuous and any fiber $f^{-1}(y)$, where $y\in Y$, is a compact subspace of $X$.

Let $X$ be a space. A {\em zero--set} of $X$ is a set of the form $\mbox{Z}(f)=f^{-1}(0)$ where $f:X\rightarrow[0,1]$ is continuous. Any set of the form $X\backslash Z$, where $Z$ is a zero--set of $X$, is called a {\em cozero--set} of $X$. We denote the set of all zero--sets of $X$ by $\mbox{Z}(X)$ and the set of all cozero--sets of $X$ by $\mbox{Coz}(X)$.

Let $X$ be a completely regular space. The {\em Stone--\v{C}ech compactification} of $X$, denoted by $\beta X$, is a compactification of $X$ characterized among all compactifications of $X$ by either of the following properties:
\begin{itemize}
  \item Every continuous mapping from $X$ to $[0,1]$ is continuously extendible over $\beta X$.
  \item Every continuous mapping from $X$ to a compact space is continuously extendible over $\beta X$.
\end{itemize}
If $Y$ is a completely regular space and $f:X\rightarrow Y$ is continuous, then there exists a (unique) continuous extension of $f$ to a mapping $f_\beta:\beta X\rightarrow\beta Y$.

\section{${\mathscr P}$--connected spaces; the definition}\label{OJ}

\begin{definition}\label{RUA}
Let $X$ be a space and let ${\mathscr P}$ be a topological property. A {\em ${\mathscr P}$--separation} for $X$ is a pair $C,D\in\mbox{Coz}(X)$ such that
\begin{itemize}
  \item $C$ and $D$ are disjoint.
  \item $\mbox{cl}_XC$ and $\mbox{cl}_XD$ are both non--${\mathscr P}$.
  \item $X\backslash(C\cup D)\subseteq E$ for some $E\in\mbox{Coz}(X)$ such that $\mbox{cl}_XE$ has ${\mathscr P}$.
\end{itemize}
The space $X$ is said to be {\em ${\mathscr P}$--disconnected} (or {\em disconnected modulo ${\mathscr P}$}) if there exists a ${\mathscr P}$--separation for it. The space $X$ is said to be {\em ${\mathscr P}$--connected} (or {\em connected modulo ${\mathscr P}$}) if it is not ${\mathscr P}$--disconnected.
\end{definition}

Observe that if ${\mathscr P}$ is the topological property of {\em being empty} then ${\mathscr P}$--connectedness coincides with connectedness in the usual sense. Thus, the notion of ${\mathscr P}$--connectedness may be considered as a generalization of connectedness.

\section{Characterization of completely regular ${\mathscr P}$--connected spaces}

In this section we characterize completely regular spaces $X$ which are connected modulo a given topological property ${\mathscr P}$. The topological property ${\mathscr P}$ here is subject to some mild requirements. Examples of such topological properties ${\mathscr P}$ are given in Example \ref{QLL}. In the case when ${\mathscr P}$ is compactness and $X$ is locally compact the characterization simplifies. This simplification is given in Theorem \ref{FKJ}. (As a corollary, we have that $\mathbb{R}^n$ is ${\mathscr P}$--connected -- with ${\mathscr P}$ being compactness -- if and only if $n\geq 2$; see Example \ref{YJHHJ}.) Other simplification occurs if ${\mathscr P}$ is taken to be pseudocompactness (see Theorem \ref{UJU}); we postpone this, however, until Section \ref{SGFF}, in which we study this special case in great detail.

Part of the results of this section either are known or are modifications of known results (see \cite{Ko3}); this includes Lemmas \ref{B}, \ref{BA27}, \ref{j1} and \ref{HFH}; the proofs are included however, for completeness of results and the reader's convenience.

The following subspace $\lambda_{\mathscr P} X$ of $\beta X$, introduced in \cite{Ko3}, plays a crucial role in our study.

\begin{definition}\label{RRA}
For a completely regular space $X$ and a topological property ${\mathscr P}$, let
\[\lambda_{\mathscr P} X=\bigcup\big\{\mbox{int}_{\beta X} \mbox{cl}_{\beta X}C:C\in\mbox{Coz}(X)\mbox{ and }\mbox{cl}_XC\mbox{ has }{\mathscr P}\big\}.\]
\end{definition}

\begin{remark}
In Definition \ref{RRA} we have
\[\lambda_{\mathscr P} X=\bigcup\big\{\mbox{int}_{\beta X} \mbox{cl}_{\beta X}Z:Z\in \mbox{Z}(X)\mbox{ has }{\mathscr P}\big\},\]
provided that ${\mathscr P}$ is closed hereditary. (See \cite{Ko4}.)
\end{remark}

\begin{lemma}\label{B}
Let $X$ be a completely regular space and let ${\mathscr P}$ be a closed hereditary topological property preserved under finite sums of closed subspaces. For any subspace $A$ of $X$, if $\mbox{\em cl}_{\beta X} A\subseteq\lambda_{\mathscr P} X$, then $\mbox{\em cl}_XA$ has ${\mathscr P}$.
\end{lemma}

\begin{proof}
By compactness of $\mbox{cl}_{\beta X} A$ and the definition of $\lambda_{\mathscr P} X$ we have
\begin{equation}\label{OB}
\mbox{cl}_{\beta X} A\subseteq\mbox{int}_{\beta X}\mbox {cl}_{\beta X}C_1\cup\cdots\cup\mbox{int}_{\beta X}\mbox {cl}_{\beta X}C_n
\end{equation}
for some $C_1,\ldots,C_n\in\mbox{Coz}(X)$ such that each $\mbox{cl}_XC_1,\ldots,\mbox{cl}_XC_n$ has ${\mathscr P}$. Intersecting both sides of (\ref{OB}) with $X$ yields
\[\mbox{cl}_XA\subseteq\mbox{cl}_XC_1\cup\cdots\cup\mbox{cl}_XC_n=B.\]
Note that $B$ has ${\mathscr P}$, as it is the finite union of its closed subspaces each with ${\mathscr P}$. Therefore $\mbox{cl}_X A$ has ${\mathscr P}$, as it is closed in $B$.
\end{proof}

We will use the following simple observation in a number of places; we record it here for convenience.

\begin{lemma}\label{LKG}
Let $X$ be a completely regular space and let $f:X\rightarrow[0,1]$ be continuous. If $0<r<1$ then
\[f_\beta^{-1}\big[[0,r)\big]\subseteq\mbox{\em int}_{\beta X}\mbox{\em cl}_{\beta X}f^{-1}\big[[0,r)\big].\]
\end{lemma}

\begin{proof}
Note that
\[f_\beta^{-1}\big[[0,r)\big]\subseteq\mbox{cl}_{\beta X}f_\beta^{-1}\big[[0,r)\big]=\mbox{cl}_{\beta X}\big(X\cap f_\beta^{-1}\big[[0,r)\big]\big)=\mbox{cl}_{\beta X}f^{-1}\big[[0,r)\big].\]
\end{proof}

\begin{lemma}\label{BA27}
Let $X$ be a completely regular space and let ${\mathscr P}$ be a topological property. Suppose that $Z\subseteq C$, where $Z\in\mbox{\em Z}(X)$, $C\in \mbox{\em Coz}(X)$ and $\mbox{\em cl}_XC$ has ${\mathscr P}$. Then $\mbox{\em cl}_{\beta X}Z\subseteq\lambda_{\mathscr P} X$.
\end{lemma}

\begin{proof}
The zero--sets $Z$ and $X\backslash C$ of $X$, are disjoint, and thus, are completely separated in $X$. Let $f:X\rightarrow[0,1]$ be continuous with $f|Z\equiv0$ and $f|(X\backslash C)\equiv1$. Using Lemma \ref{LKG} we have
\[\mbox{cl}_{\beta X}Z\subseteq \mbox{Z}(f_\beta)\subseteq f_\beta^{-1}\big[[0,1/2)\big]\subseteq\mbox{int}_{\beta X}\mbox{cl}_{\beta X}f^{-1}\big[[0,1/2)\big]\subseteq\mbox{int}_{\beta X}\mbox{cl}_{\beta X} C\subseteq\lambda_{\mathscr P}X.\]
\end{proof}

Let $X$ be a completely regular space. For an open subspace $U$ of $X$, the {\em extension of $U$ to $\beta X$} is defined to be
\[\mbox {Ex}_XU=\beta X\backslash\mbox {cl}_{\beta X}(X\backslash U).\]
The following lemma is well known. (See Lemma 7.1.13 of \cite{E} or  Lemma 3.1 of \cite{vD}.)

\begin{lemma}\label{9}
Let $X$ be a completely regular space and let $U$ and $V$ be open subspaces of $X$. Then
\begin{itemize}
\item[\rm(1)] $X\cap\mbox{\em Ex}_XU=U$, and thus $\mbox{\em cl}_{\beta X}\mbox{\em Ex}_XU=\mbox{\em cl}_{\beta X}U$.
\item[\rm(2)] $\mbox{\em Ex}_X(U\cap V)=\mbox{\em Ex}_XU\cap\mbox{\em Ex}_XV$.
\end{itemize}
\end{lemma}

The following lemma is proved by E.G. Skljarenko in \cite{S}. It is rediscovered by E.K. van Douwen in \cite{vD}.

\begin{lemma}[Skljarenko \cite{S} and van Douwen  \cite{vD}]\label{10}
Let $X$ be a completely regular space and let $U$ be an open subspace of $X$. Then
\[\mbox{\em bd}_{\beta X}\mbox {\em Ex}_XU=\mbox{\em cl}_{\beta X}\mbox{\em bd}_XU.\]
\end{lemma}

\begin{lemma}\label{j1}
Let $X$ be a completely regular space and let ${\mathscr P}$ be a topological property. Let $U$ be an open subspace of $X$ such that $\mbox{\em bd}_XU\subseteq Z\subseteq C$, where $Z\in\mbox{\em Z}(X)$, $C\in\mbox{\em Coz}(X)$ and $\mbox{\em cl}_XC$ has ${\mathscr P}$. Then
\[\mbox{\em cl}_{\beta X}U\backslash\lambda_{\mathscr P} X=\mbox{\em Ex}_X U\backslash\lambda_{\mathscr P} X.\]
\end{lemma}

\begin{proof}
By Lemma \ref{BA27} we have $\mbox{cl}_{\beta X}Z\subseteq\lambda_{\mathscr P} X$. The result then follows, as by Lemmas \ref{9} and \ref{10} we have
\[\mbox{cl}_{\beta X}U=\mbox{cl}_{\beta X}\mbox{Ex}_XU=\mbox{Ex}_XU\cup\mbox{bd}_{\beta X}\mbox{Ex}_XU=\mbox{Ex}_XU\cup\mbox{cl}_{\beta X}\mbox{bd}_XU\]
and $\mbox{cl}_{\beta X}\mbox{bd}_X U\subseteq\mbox{cl}_{\beta X}Z$.
\end{proof}

We are now ready to prove the following main result of this section.

\begin{theorem}\label{TTES}
Let $X$ be a completely regular space and let ${\mathscr P}$ be a closed hereditary topological property preserved under finite sums of closed subspaces. The following are equivalent:
\begin{itemize}
\item[\rm(1)] $X$ is ${\mathscr P}$--connected.
\item[\rm(2)] $\beta X\backslash\lambda_{\mathscr P} X$ is connected.
\end{itemize}
\end{theorem}

\begin{proof}
(1) {\em  implies} (2). Suppose that $\beta X\backslash\lambda_{\mathscr P} X$ is disconnected. We show that $X$ is then ${\mathscr P}$--disconnected. Let $G$ and $H$ be a separation for $\beta X\backslash\lambda_{\mathscr P} X$. Since $G$ and $H$ are closed in $\beta X\backslash\lambda_{\mathscr P} X$ and the latter is compact, as it is closed in $\beta X$ (note that $\lambda_{\mathscr P} X$ is open in $\beta X$ by its definition) $G$ and $H$ are compact and thus closed in $\beta X$. By normality of $\beta X$ there exists a continuous $f:\beta X\rightarrow[0,1]$ with $f|G\equiv0$ and $f|H\equiv1$. Let
\[C=X\cap f^{-1}\big[[0,1/2)\big]\;\;\;\;\mbox{ and }\;\;\;\;D=X\cap f^{-1}\big[(1/2,1]\big].\]
Then $C,D\in\mbox{Coz}(X)$ and $C$ and $D$ are disjoint. We prove that $\mbox{cl}_XC$ is non--${\mathscr P}$; the proof that $\mbox{cl}_XD$ is non--${\mathscr P}$ is analogous. Suppose to the contrary that $\mbox{cl}_XC$ has ${\mathscr P}$. Then, using Lemma \ref{LKG} we have
\[G\subseteq f^{-1}\big[[0,1/2)\big]\subseteq\mbox{int}_{\beta X}\mbox{cl}_{\beta X}\big(X\cap f^{-1}\big[[0,1/2)\big]\big)=\mbox{int}_{\beta X}\mbox{cl}_{\beta X} C\subseteq\lambda_{\mathscr P} X,\]
which is a contradiction. Note that
\[X\backslash (C\cup D)=X\cap f^{-1}(1/2)\subseteq X\cap f^{-1}\big((1/3, 2/3)\big)=E\in\mbox{Coz}(X)\]
and $\mbox{cl}_XE$ has ${\mathscr P}$ by Lemma \ref{B}, as
\[\mbox{cl}_{\beta X}E\subseteq f^{-1}\big([1/3, 2/3]\big)\subseteq\lambda_{\mathscr P} X.\]
Thus the pair $C$ and $D$ is a ${\mathscr P}$--separation for $X$, that is, $X$ is ${\mathscr P}$--disconnected.

(2) {\em  implies} (1). Suppose that $X$ is ${\mathscr P}$--disconnected. We show that $\beta X\backslash\lambda_{\mathscr P} X$ is disconnected. Let $U$ and $V$ be a ${\mathscr P}$--separation for $X$ and let $W\in\mbox{Coz}(X)$ be such that $\mbox{cl}_XW$ has ${\mathscr P}$ and $Z=X\backslash(U\cup V)\subseteq W$. Note that $Z\in\mbox{Z}(X)$. By Lemma \ref{BA27} we have $\mbox{cl}_{\beta X}Z\subseteq\lambda_{\mathscr P} X$ and then, since $X=U\cup V\cup Z$, it follows that
\[\beta X\backslash\lambda_{\mathscr P} X=(\mbox{cl}_{\beta X}U\backslash\lambda_{\mathscr P} X)\cup(\mbox{cl}_{\beta X}V\backslash\lambda_{\mathscr P} X).\]
Since
\[\mbox{bd}_XU=\mbox{cl}_XU\cap(X\backslash U)=\mbox{cl}_XU\cap(V\cup Z)=\mbox{cl}_XU\cap Z\subseteq Z,\]
by Lemma \ref{j1} we have
\[A=\mbox{cl}_{\beta X}U\backslash\lambda_{\mathscr P} X=\mbox{Ex}_X U\backslash\lambda_{\mathscr P} X.\]
Note that $A$ is non--empty, as otherwise, we have $\mbox{cl}_{\beta X}U\subseteq\lambda_{\mathscr P} X$, which by Lemma \ref{B} is impossible, as $\mbox{cl}_XU$ is non--${\mathscr P}$.
Similarly, if we let
\[B=\mbox{cl}_{\beta X}V\backslash\lambda_{\mathscr P} X=\mbox{Ex}_X V\backslash\lambda_{\mathscr P} X\]
then $B$ is non--empty. Observe that both $A$ and $B$ are closed subspaces of $\beta X\backslash\lambda_{\mathscr P} X$. To conclude the proof, note that using Lemma \ref{9} we have
\[A\cap B\subseteq\mbox{Ex}_X U\cap\mbox{Ex}_X V=\mbox{Ex}_X (U\cap V)=\emptyset.\]
Thus the pair $A$ and $B$ is a separation for $\beta X\backslash\lambda_{\mathscr P} X$, that is, the latter is disconnected.
\end{proof}

\begin{example}\label{QLL}
The list of topological properties ${\mathscr P}$ satisfying the assumption of Theorem \ref{TTES} is quite long and include almost all important covering properties; among them are: compactness, countable compactness (more generally, $[\theta,\kappa]$--compactness), the Lindel\"{o}f property (more generally, the $\mu$--Lindel\"{o}f property), paracompactness, metacompactness, countable paracompactness, subparacompactness, submetacompactness (or $\theta$--refinability), the $\sigma$--para--Lindel\"{o}f property and also $\alpha$--boundedness. (See \cite{Bu} and \cite{Steph} for definitions. That these topological properties -- except for the last one -- are closed hereditary and preserved under finite sums of closed subspaces, follows from Theorems 7.1, 7.3 and 7.4 of \cite{Bu}; for $\alpha$--boundedness, this directly follows from its definition. Recall that a Hausdorff space $X$ is called {\em $\alpha$--bounded}, where $\alpha$ is an infinite cardinal, if every subspace of $X$ of cardinality $\leq\alpha$ has compact closure in $X$.)
\end{example}

In Theorem \ref{TTES}, it is worth to know when $\beta X\backslash\lambda_{\mathscr P} X\subseteq\beta X\backslash X$, or equivalently, when $X\subseteq\lambda_{\mathscr P} X$. This is the subject matter of the following lemma.

\begin{lemma}\label{HFH}
Let $X$ be a completely regular space and let ${\mathscr P}$ be a closed hereditary topological property.  Then $X\subseteq\lambda_{\mathscr P} X$ if and only if $X$ is locally--${\mathscr P}$.
\end{lemma}

\begin{proof}
Suppose that $X$ is locally--${\mathscr P}$. Let $x\in X$ and let $U$ be an open neighborhood of $x$ in $X$ such that $\mbox{cl}_XU$ has ${\mathscr P}$. Let $f:X\rightarrow[0,1]$ be continuous with $f(x)=0$ and $f|(X\backslash U)\equiv 1$. Let
\[C=f^{-1}\big[[0,1/2)\big]\in\mbox{Coz}(X).\]
Then $C\subseteq U$ and thus $\mbox{cl}_XC$ has ${\mathscr P}$, as it is closed in $\mbox{cl}_XU$. Using Lemma \ref{LKG} we have
\begin{eqnarray*}
x\in f_\beta^{-1}\big[[0,1/2)\big]\subseteq\mbox{int}_{\beta X}\mbox{cl}_{\beta X}f^{-1}\big[[0,1/2)\big]=\mbox{int}_{\beta X}\mbox{cl}_{\beta X}C\subseteq\lambda_{\mathscr P} X.
\end{eqnarray*}
Therefore $X\subseteq\lambda_{\mathscr P} X$.

For the converse, suppose that $X\subseteq\lambda_{\mathscr P} X$. Let $x\in X$. Then $x\in\lambda_{\mathscr P} X$ and thus $x\in\mbox{int}_{\beta X}\mbox{cl}_{\beta X}C$ for some $C\in\mbox{Coz}(X)$ such that $\mbox{cl}_XC$ has ${\mathscr P}$. Let
\[V=X\cap\mbox{int}_{\beta X}\mbox{cl}_{\beta X}C.\]
Then $V$ is an open neighborhood of $x$ in $X$. Since $V\subseteq\mbox{cl}_XC$, the closure $\mbox{cl}_XV$ has ${\mathscr P}$, as it is closed in $\mbox{cl}_XC$. Therefore $X$ is locally--${\mathscr P}$.
\end{proof}

The following is a special case of Theorem \ref{TTES}.

\begin{theorem}\label{FKJ}
Let $X$ be a locally compact Hausdorff space and let ${\mathscr P}$ be compactness. The following are equivalent:
\begin{itemize}
\item[\rm(1)] $X$ is ${\mathscr P}$--connected.
\item[\rm(2)] $\beta X\backslash X$ is connected.
\end{itemize}
\end{theorem}

\begin{proof}
Note that $\lambda_{\mathscr P} X\subseteq X$, as if $C\in\mbox{Coz}(X)$ has compact closure, then $\mbox{cl}_{\beta X}C\subseteq\mbox{cl}_XC\subseteq X$. Also, $X\subseteq\lambda_{\mathscr P} X$ by Lemma \ref{HFH}, as $X$ is locally compact. Therefore $\lambda_{\mathscr P} X=X$. The theorem now follows from Theorem \ref{TTES}.
\end{proof}

\begin{example}\label{YJHHJ}
It is known that $\beta\mathbb{R}^n\backslash\mathbb{R}^n$ is connected if and only if $n\geq 2$. (See 6.10 and Exercise 6.L of \cite{GJ}.) Thus, if ${\mathscr P}$ is compactness, then it follows from Theorem \ref{FKJ} that $\mathbb{R}^n$ is ${\mathscr P}$--connected if and only if $n\geq 2$.
\end{example}

\section{Unions of ${\mathscr P}$--connected subspaces}

The union of connected subspaces with non--empty intersection is connected. This section deals with generalizations of this well--known result in the context of ${\mathscr P}$--connectedness.

\begin{lemma}\label{ERD}
Let $X$ be a space and let ${\mathscr P}$ be a closed hereditary topological property. Let $C$ and $D$ be a ${\mathscr P}$--separation for $X$. Then for every closed ${\mathscr P}$--connected subspace $A$ of $X$ either $\mbox{\em cl}_X(A\cap C)$ or $\mbox{\em cl}_X(A\cap D)$ has ${\mathscr P}$.
\end{lemma}

\begin{proof}
Let $E\in\mbox{Coz}(X)$ be such that $\mbox{cl}_XE$ has ${\mathscr P}$ and $Z=X\backslash(C\cup D)\subseteq E$. Let $A$ be a closed ${\mathscr P}$--connected subspace of $X$. Note that $A\cap C, A\cap D\in\mbox{Coz}(A)$ are disjoint. Also
\[A\backslash\big((A\cap C)\cup (A\cap D)\big)=A\cap Z\subseteq A\cap E\in\mbox{Coz}(A)\]
and $\mbox{cl}_A(A\cap E)$ has ${\mathscr P}$, as $\mbox{cl}_A(A\cap E)=\mbox{cl}_X(A\cap E)$ is closed in $\mbox{cl}_XE$ and the latter has ${\mathscr P}$. Since $A$ is ${\mathscr P}$--connected, either
\[\mbox{cl}_A(A\cap C)=\mbox{cl}_X(A\cap C)\;\;\;\;\mbox{ or }\;\;\;\;\mbox{cl}_A(A\cap D)=\mbox{cl}_X(A\cap D)\]
has ${\mathscr P}$.
\end{proof}

\begin{theorem}\label{HGHG}
Let $X$ be a space and let ${\mathscr P}$ be a closed hereditary topological property preserved under finite sums of closed subspaces. Let
\[X=X_1\cup\cdots\cup X_n,\]
where each $X_1,\ldots,X_n$ is a closed ${\mathscr P}$--connected subspace of $X$. If there exists a ${\mathscr P}$--connected non--${\mathscr P}$ closed subspace $A$ of $X$ with
\[A\subseteq X_1\cap\cdots\cap X_n\]
then $X$ is ${\mathscr P}$--connected.
\end{theorem}

\begin{proof}
Let $A$ be a ${\mathscr P}$--disconnected non--${\mathscr P}$ closed subspace of $X$ such that $A\subseteq X_1\cap\cdots\cap X_n$. Suppose to the contrary that $X$ is ${\mathscr P}$--disconnected. Let $C$ and $D$ be a ${\mathscr P}$--separation for $X$ and let $E\in\mbox{Coz}(X)$ be such that $\mbox{cl}_XE$ has ${\mathscr P}$ and $Z=X\backslash(C\cup D)\subseteq E$. By Lemma \ref{ERD} either $\mbox{cl}_X(A\cap C)$ or $\mbox{cl}_X(A\cap D)$, say the latter, has ${\mathscr P}$. Note that $A\cap Z$ has ${\mathscr P}$, as it is closed in $\mbox{cl}_XE$. Since
\[A=\mbox{cl}_X(A\cap C)\cup\mbox{cl}_X(A\cap D)\cup(A\cap Z)\]
it follows that $\mbox{cl}_X(A\cap C)$ is non--${\mathscr P}$; as otherwise, $A$ has ${\mathscr P}$, as it is the finite union of its closed ${\mathscr P}$--subspaces. Let $i=1,\ldots,n$ be fixed. Since $X_i$ is ${\mathscr P}$--connected and closed in $X$, by Lemma \ref{ERD}, either $\mbox{cl}_X(X_i\cap C)$ or $\mbox{cl}_X(X_i\cap D)$ has ${\mathscr P}$. But $\mbox{cl}_X(X_i\cap C)$ is non--${\mathscr P}$, as its closed subspace $\mbox{cl}_X(A\cap C)$ is so. Now
\[\mbox{cl}_XD=\mbox{cl}_X(X_1\cap D)\cup\cdots\cup\mbox{cl}_X(X_n\cap D)\]
has ${\mathscr P}$, as it is the finite union of its closed ${\mathscr P}$--subspaces. This contradiction proves the theorem.
\end{proof}

\begin{example}\label{DHJ}
Topological properties ${\mathscr P}$ satisfying the assumption of Theorem \ref{HGHG} are identical to those satisfying the assumption of Theorem \ref{TTES}. Thus the list of such topological properties includes all those introduced in Example \ref{QLL}.
\end{example}

\begin{theorem}\label{TYHG}
Let $X$ be a ${\mathscr Q}$--space, where ${\mathscr Q}$ is a closed hereditary topological property. Let ${\mathscr P}$ be a closed hereditary topological property preserved under countable sums of closed subspaces. Suppose that every ${\mathscr Q}$--space with a dense ${\mathscr P}$--subspace has ${\mathscr P}$. Let
\[X=X_1\cup X_2\cup\cdots,\]
where each $X_1,X_2,\ldots$ is a closed ${\mathscr P}$--connected subspace of $X$. If there exists a ${\mathscr P}$--connected non--${\mathscr P}$ closed subspace $A$ of $X$ with
\[A\subseteq X_1\cap X_2\cap\cdots\]
then $X$ is ${\mathscr P}$--connected.
\end{theorem}

\begin{proof}
Let $A$ be a ${\mathscr P}$--connected non--${\mathscr P}$ closed subspace of $X$ such that $A\subseteq X_1\cap X_2\cap\cdots$. Suppose to the contrary that $X$ is ${\mathscr P}$--disconnected. Let $C$ and $D$ be a ${\mathscr P}$--separation for $X$. By Lemma \ref{ERD} either $\mbox{cl}_X(A\cap C)$ or $\mbox{cl}_X(A\cap D)$, say the latter, has ${\mathscr P}$. Arguing as in the proof of Theorem \ref{HGHG} it follows that $\mbox{cl}_X(X_i\cap D)$ has ${\mathscr P}$ for each $i=1,2,\ldots$. Then
\[H=\mbox{cl}_X(X_1\cap D)\cup\mbox{cl}_X(X_2\cap D)\cup\cdots\]
has ${\mathscr P}$, as it is the countable union of its closed ${\mathscr P}$--subspaces. Note that $\mbox{cl}_XD$ has ${\mathscr Q}$, as it is closed in $X$. Since $\mbox{cl}_XD$ contains $H$ as a dense subspace, it then follows that $\mbox{cl}_XD$ has ${\mathscr P}$. This is a contradiction.
\end{proof}

\begin{example}\label{HJ}
Let ${\mathscr Q}$ be paracompactness and let ${\mathscr P}$ be the Lindel\"{o}f property. Then ${\mathscr Q}$ and ${\mathscr P}$ are both closed hereditary (see Theorems 3.8.4 and 5.1.28 of \cite{E}) and ${\mathscr P}$ is trivially preserved under countable sums of closed subspaces. Also, every ${\mathscr Q}$--space with a dense ${\mathscr P}$--subspace has ${\mathscr P}$. (See Theorem 5.1.25 of \cite{E}.) Thus, the pair ${\mathscr Q}$ and ${\mathscr P}$ satisfies the requirement of Theorem \ref{TYHG}.
\end{example}

\section{Images of ${\mathscr P}$--connected spaces}

The continuous image of any connected spaces is connected. This section deals with generalizations of this well--known result in the context of ${\mathscr P}$--connectedness.

\begin{theorem}\label{YG}
Let ${\mathscr P}$ be a closed hereditary topological property both invariant and inverse invariant under perfect continuous surjections. Then the perfect continuous image of any ${\mathscr P}$--connected space is ${\mathscr P}$--connected.
\end{theorem}

\begin{proof}
Let $f:X\rightarrow Y$ be a perfect continuous surjection. Suppose that $Y$ is ${\mathscr P}$--disconnected and let $C$ and $D$ be a ${\mathscr P}$--separation for $Y$. We show that $X$ is ${\mathscr P}$--disconnected by showing that the pair $f^{-1}[C]$ and $f^{-1}[D]$ forms a ${\mathscr P}$--separation for $X$. Let $E\in\mbox{Coz}(Y)$ be such that $\mbox{cl}_YE$ has ${\mathscr P}$ and $Y\backslash(C\cup D)\subseteq E$. Note that $f^{-1}[C],f^{-1}[D]\in\mbox{Coz}(X)$ and that they are disjoint. Suppose to the contrary that the closure in $X$ of either $f^{-1}[C]$ or $f^{-1}[D]$, say the latter, has ${\mathscr P}$. Since \[f|\mbox{cl}_Xf^{-1}[D]:\mbox{cl}_Xf^{-1}[D]\rightarrow f\big[\mbox{cl}_Xf^{-1}[D]\big]\]
is perfect continuous and surjective, and ${\mathscr P}$ is invariant under perfect continuous surjections, $f[\mbox{cl}_Xf^{-1}[D]]$ has ${\mathscr P}$. Note that since $f$ is surjective we have $D=f[f^{-1}[D]]$ and since $f$ is closed, $f[\mbox{cl}_Xf^{-1}[D]]$ is closed in $Y$. This implies that $\mbox{cl}_YD$ has ${\mathscr P}$, as it is contained in $f[\mbox{cl}_Xf^{-1}[D]]$ as a closed subspace. This contradiction shows that $\mbox{cl}_Xf^{-1}[D]$ is non--${\mathscr P}$. Next, note that  \[X\backslash\big(f^{-1}[C]\cup f^{-1}[D]\big)=f^{-1}\big[Y\backslash(C\cup D)\big]\subseteq f^{-1}[E]\]
and $f^{-1}[E]\in\mbox{Coz}(X)$. Thus, to complete the proof we need to show that $\mbox{cl}_Xf^{-1}[E]$ has ${\mathscr P}$. Note that
\[f|f^{-1}[\mbox{cl}_YE]:f^{-1}[\mbox{cl}_YE]\rightarrow\mbox{cl}_YE\]
is perfect continuous and it is surjective. (The latter is because $f[f^{-1}[\mbox{cl}_YE]]$ is closed in $Y$, as $f$ is closed, and it contains $E=f[f^{-1}[E]]$.) Since $\mbox{cl}_YE$ has ${\mathscr P}$ and ${\mathscr P}$ is inverse invariant under perfect continuous surjections, $f^{-1}[\mbox{cl}_YE]$ has ${\mathscr P}$. But then $\mbox{cl}_Xf^{-1}[E]$ has ${\mathscr P}$, as it is contained in $f^{-1}[\mbox{cl}_YE]$ as a closed subspace.
\end{proof}

\begin{example}\label{JJHJ}
Topological properties ${\mathscr P}$ introduced in Example \ref{QLL} all satisfy the assumption of Theorem \ref{YG}. The fact that they are invariant under perfect continuous surjections follows from Theorems 5.1 and 5.5 of \cite{Bu} and Exercise 5.2.G of \cite{E}. Also, by (modification of) Theorem 3.7.24 and Exercise 5.2.G of \cite{E} and Theorem 5.9 of \cite{Bu} it follows that these topological properties are all inverse invariant under perfect continuous surjections. For the case of $\alpha$--boundedness, note that for a perfect continuous surjective $f:X\rightarrow Y$, when $Y$ is $\alpha$--bounded, if a subspace $A$ of $X$ has cardinality $\leq\alpha$, then $f[A]$ has cardinality $\leq\alpha$, and thus $\mbox{cl}_Y f[A]$ is compact. But since
\[A\subseteq f^{-1}\big[f[A]\big]\subseteq f^{-1}\big[\mbox{cl}_Y f[A]\big]\]
and the latter is compact (note that in any perfect continuous mapping the inverse image of each compact subspaces of the codomain is compact; see Theorem 3.7.2 of \cite{E}), its closed subspace $\mbox{cl}_X A$ also is compact, that is, $X$ is $\alpha$--bounded.
\end{example}

\section{Spaces with dense ${\mathscr P}$--connected subspaces}\label{UOOJ}

Any space containing a dense connected subspaces is connected. This does not hold in general in the context of ${\mathscr P}$--connectedness. We show this in the following example.

\begin{example}\label{OOJ}
Let ${\mathscr P}$ be compactness. Let $X=\mathbb{R}\times[0,1]$, considered as a subspace of $\mathbb{R}^2$, and let $A=\mathbb{R}\times(0,1)$. Then $A$ is dense in $X$. Note that $A$ is homeomorphic to $\mathbb{R}^2$ and is thus ${\mathscr P}$--connected by Example \ref{YJHHJ}. However, $X$ is ${\mathscr P}$--disconnected, as the pair $(-\infty,0)\times[0,1]$ and $(1,\infty)\times[0,1]$ constitutes a ${\mathscr P}$--separation for it.
\end{example}

\section{Products of ${\mathscr P}$--connected spaces}\label{UOOJ}

Products of connected spaces are connected. We do not know to what extent this remains true in the more general context of ${\mathscr P}$--connectedness. We formally state this below as an open question.

\begin{question}\label{GFF}
Under what conditions is the product of ${\mathscr P}$--connected spaces ${\mathscr P}$--connected?
\end{question}

\section{Connectedness modulo pseudocompactness}\label{SGFF}

This section deals with ${\mathscr P}$--connectedness in the case when ${\mathscr P}$ is pseudocompactness.

\subsubsection*{\bf{Characterization of completely regular pseudocompactness--connected spaces}} In this part we prove a result, analogous to Theorem \ref{TTES}, characterizing completely regular pseudocompactness--connected spaces. But before we proceed, we need to determine $\lambda_{\mathscr P}X$ in this case. This is actually done in \cite{Ko4} (see also \cite{Ko7}); we give the details here for completeness of results.

For any completely regular space $X$ we denote by $\upsilon X$ the Hewitt realcompactification of $X$. One may assume that $\upsilon X\subseteq\beta X$.

The following result is due to  A.W. Hager and D.G. Johnson in \cite{HJ}; a direct proof may be found in \cite{C}. (See also Theorem 11.24 of \cite{We}.)

\begin{lemma}[Hager--Johnson \cite{HJ}]\label{A}
Let $U$ be an open subspace of the completely regular space $X$. If $\mbox{\em cl}_{\upsilon X} U$ is compact then $\mbox{\em cl}_X U$ is pseudocompact.
\end{lemma}

Observe, in the proof of the following, that realcompactness is closed hereditary, a space having a pseudocompact dense subspace is pseudocompact, and that realcompact pseudocompact spaces are compact; see Theorems 3.11.1 and 3.11.4 of \cite{E}.

\begin{lemma}\label{HGA}
Let $U$ be an open subspace of the completely regular space $X$. Then $\mbox{\em cl}_{\beta X} U\subseteq\upsilon X$ if and only if $\mbox{\em cl}_X U$ is pseudocompact.
\end{lemma}

\begin{proof}
The first half follows from Lemma \ref{A}. For the second half, note that if $A=\mbox{cl}_X U$ is pseudocompact then so is its closure $\mbox{cl}_{\upsilon X}A$. But  $\mbox{cl}_{\upsilon X} A$ is also realcompact, as it is closed in $\upsilon X$, and thus it is compact. Therefore $\mbox{cl}_{\beta X} A\subseteq\mbox{cl}_{\upsilon X} A$.
\end{proof}

\begin{lemma}\label{PTF}
Let ${\mathscr P}$ be pseudocompactness. Let $X$ be a completely regular space. Then
\[\lambda_{\mathscr P}X=\mbox{\em int}_{\beta X}\upsilon X.\]
\end{lemma}

\begin{proof}
If $C\in\mbox{Coz}(X)$ has pseudocompact closure in $X$ then $\mbox{cl}_{\beta X} C\subseteq\upsilon X$, by Lemma \ref{HGA}, and then $\mbox{int}_{\beta X}\mbox{cl}_{\beta X} C\subseteq\mbox{int}_{\beta X}\upsilon X$. Thus $\lambda_{\mathscr P}X\subseteq\mbox{int}_{\beta X}\upsilon X$.

For the reverse inclusion, let $t\in\mbox{int}_{\beta X}\upsilon X$. Let $f:\beta X\rightarrow[0,1]$ be continuous with
\[f(t)=0\;\;\;\;\mbox{ and }\;\;\;\;f|(\beta X\backslash\mbox{int}_{\beta X}\upsilon X)\equiv 1.\]
Then
\[C=X\cap f^{-1}\big[[0,1/2)\big]\in\mbox{Coz}(X)\]
and $t\in\mbox{int}_{\beta X}\mbox{cl}_{\beta X} C$ by Lemma \ref{LKG}. (Note that $(f|X)_\beta=f$, as they coincide on the dense subspace $X$ of $\beta X$.) Also, $\mbox{cl}_X C$ is pseudocompact by Lemma \ref{HGA}, as
\[\mbox{cl}_{\beta X}C\subseteq f^{-1}\big[[0,1/2]\big]\subseteq\upsilon X.\]
Therefore $\mbox{int}_{\beta X}\upsilon X\subseteq\lambda_{\mathscr P}X$.
\end{proof}

The following is the main result of this part.

\begin{theorem}\label{UJU}
Let $X$ be a completely regular space. The following are equivalent:
\begin{itemize}
\item[\rm(1)] $X$ is pseudocompactness--connected.
\item[\rm(2)] $\mbox{\em cl}_{\beta X}(\beta X\backslash\upsilon X)$ is connected.
\end{itemize}
\end{theorem}

\begin{proof}
Observe that pseudocompactness is hereditary with respect to regular closed subspaces (see Problem 3.10.F of \cite{E}) and it is preserved under finite sums of closed subspaces. Thus, Lemma \ref{B} holds true if ${\mathscr P}$ is pseudocompactness and $A$ is a regular closed subspace of $X$. The proof is now analogous to the one given for Theorem \ref{TTES}.
\end{proof}

\subsubsection*{\bf{Images of pseudocompactness--connected spaces}} In this part we study images of pseudocompactness--connected spaces under certain classes of mappings.

Let $X$ and $Y$ be completely regular spaces. A continuous mapping $f:X\rightarrow Y$ is said to be {\em hyper--real} if $f_\beta[\beta X\backslash\upsilon X]\subseteq\beta Y\backslash\upsilon Y$. Hyper--real mappings are defined by R.L. Blair in the unpublished manuscript \cite{Bl} and provide the appropriate tool for the
study of preservation of realcompactness and inverse preservation of pseudocompactness. It is known that every perfect open continuous surjection between completely regular spaces is hyper--real. (See Corollaries 15.14 and 17.19 of \cite{We}.)

Recall that the Hewitt realcompactification of a completely regular space $X$ may be expressed as the union of all cozero--sets of $\beta X$ containing $X$.

The following is a counterpart for Theorem \ref{YG}.

\begin{theorem}\label{YDG}
Every perfect open continuous image of a completely regular pseudocompactness--connected space is pseudocompactness--connected.
\end{theorem}

\begin{proof}
Let $X$ be a completely regular pseudocompactness--connected space and let $f:X\rightarrow Y$ be a perfect open continuous surjection. Note that complete regularity is invariant under perfect open continuous surjections (see page 512 of \cite{E}); thus $Y$ is completely regular. Observe that if $Z\in\mbox{Z}(\beta Y)$ and $Z\cap Y=\emptyset$ then $f_\beta^{-1}[Z]\cap X=\emptyset$, as
\[f_\beta^{-1}[Z]\cap X\subseteq f_\beta^{-1}[Z]\cap f_\beta^{-1}[Y].\]
Therefore
\begin{eqnarray*}
f_\beta^{-1}[\beta Y\backslash\upsilon Y]&=&f_\beta^{-1}\Big[\bigcup\big\{Z\in\mbox{Z}(\beta Y):Z\cap Y=\emptyset\big\}\Big]\\&=&\bigcup\big\{f_\beta^{-1}[Z]:Z\in \mbox{Z}(\beta Y)\mbox{ and }Z\cap Y=\emptyset\big\}\\&\subseteq&\bigcup\big\{S\in\mbox{Z}(\beta X):S\cap X=\emptyset\big\}=\beta X\backslash\upsilon X.
\end{eqnarray*}
On the other hand, since $f$ is hyper--real, we have
\[\beta X\backslash\upsilon X\subseteq f_\beta^{-1}\big[f_\beta[\beta X\backslash\upsilon X]\big]\subseteq f_\beta^{-1}[\beta Y\backslash\upsilon Y].\]
Thus $f_\beta^{-1}[\beta Y\backslash\upsilon Y]=\beta X\backslash\upsilon X$. We have
\[f_\beta\big[\mbox{cl}_{\beta X}(\beta X\backslash\upsilon X)\big]\subseteq\mbox{cl}_{\beta Y}\big(f_\beta[\beta X\backslash\upsilon X]\big)=\mbox{cl}_{\beta Y}\big(f_\beta \big[f_\beta^{-1}[\beta Y\backslash\upsilon Y]\big]\big)\subseteq\mbox{cl}_{\beta Y}(\beta Y\backslash\upsilon Y).\]
Also, since $f_\beta$ is surjective (as $f$ is so) we have
\[\beta Y\backslash\upsilon Y=f_\beta \big[f_\beta^{-1}[\beta Y\backslash\upsilon Y]\big]=f_\beta[\beta X\backslash\upsilon X]\subseteq f_\beta\big[\mbox{cl}_{\beta X}(\beta X\backslash\upsilon X)\big]\]
and therefore
\[\mbox{cl}_{\beta Y}(\beta Y\backslash\upsilon Y)\subseteq f_\beta\big[\mbox{cl}_{\beta X}(\beta X\backslash\upsilon X)\big].\]
That is
\[f_\beta\big[\mbox{cl}_{\beta X}(\beta X\backslash\upsilon X)\big]=\mbox{cl}_{\beta Y}(\beta Y\backslash\upsilon Y).\]
Since $X$ is pseudocompactness--connected, $\mbox{cl}_{\beta X}(\beta X\backslash\upsilon X)$ is connected, by Theorem \ref{UJU}, and thus, so is its continuous image $\mbox{cl}_{\beta Y}(\beta Y\backslash\upsilon Y)$. From the latter, it then follows from Theorem \ref{UJU} that $Y$ is pseudocompactness--connected.
\end{proof}

\bibliographystyle{amsplain}

\end{document}